\documentclass[journal,twoside,web]{ieeecolor}
\usepackage{lcsys}
\usepackage{cite}
\usepackage{amsmath,amssymb,amsfonts}
\usepackage{algorithmic}
\usepackage{graphicx}
\usepackage{textcomp}
\usepackage[utf8]{inputenc}
\usepackage{dsfont}
\usepackage{bbold}
\usepackage{tikz}
\usepackage{booktabs}
\usepackage[ruled]{algorithm2e} 
\usetikzlibrary{arrows}

\newtheorem{prp}{Proposition}
\newtheorem{lemma}{Lemma}
\newtheorem{theorem}{Theorem}

\newtheorem{ex}{Example}
\newtheorem{definition}{Definition}[section]
\newtheorem{remark}{Remark}

\DeclareMathOperator*{\argmax}{arg\,max}
\DeclareMathOperator*{\argmin}{arg\,min}
\newcommand{\mc}{\mathcal}

\newcommand{\be}{\begin{equation}}
	\newcommand{\ee}{\end{equation}}
\newcommand{\ba}{\begin{array}}
	\newcommand{\ea}{\end{array}}

\newcommand{\ov}{\overline}

\newcommand{\ds}{\displaystyle}

\newcommand{\1}{\mathbf 1}

\newcommand{\R}{\mathds R}
\renewcommand{\P}{\mathds P}
\newcommand{\E}{\mathbf E}

\newcommand{\tcb}{\textcolor{blue}}

\def\BibTeX{{\rm B\kern-.05em{\sc i\kern-.025em b}\kern-.08em
    T\kern-.1667em\lower.7ex\hbox{E}\kern-.125emX}}
\begin{document}
\title{\tcb{Bayesian Equilibria of Heterogeneous Non-Atomic 
Routing Games with Private Information}}
\author{Alexia Ambrogio, \IEEEmembership{Student Member, IEEE}, Leonardo Cianfanelli, \IEEEmembership{Member, IEEE}, Giacomo Como, \IEEEmembership{Member, IEEE}, and Paolo Frasca, \IEEEmembership{Senior Member, IEEE}
\thanks{This publication is part of project PNRR-NGEU, which has received funding from the MUR - DM 118/2023.}
\thanks{A.~Ambrogio, L.~Cianfanelli and G.~Como are with Politecnico di Torino, Department of Mathematical Sciences, 10129 Turin, Italy (e-mail: \{alexia.ambrogio, leonardo.cianfanelli, giacomo.como\}@polito.it).}
\thanks{P.~Frasca is with Univ.\ Grenoble Alpes, CNRS, Inria, Grenoble INP, GIPSA-Lab, 38000 Grenoble, France (e-mail: paolo.frasca@inria.fr).}
}
\maketitle

\begin{abstract}
We study non-atomic Bayesian routing games whereby a transportation network is shared by two types of traffic: a coordinated fleet and a mass of selfish users. The links in the network are characterized by travel time functions that depend both on the aggregate flow on the link and on a random state of the world $W$ that, in general, is not directly observable.  Rather, we assume that both the fleet coordinator and the selfish users know the prior distribution of $W$, observe (different) private messages that are correlated with $W$ and possibly among themselves, and make routing decisions based on such heterogeneous partial information.  Under the assumption that both the state of the world and the private message sets are finite, we prove the existence and uniqueness of a Bayesian equilibrium for the ensuing Bayesian routing game. 
\end{abstract}


\section{Introduction}
Understanding how traffic flows emerge from the interaction of many travelers is a central problem in transportation science. A widely adopted approach models route choice as a strategic interaction among users of the transportation network. When the number of travelers is large, transportation systems can be modeled by non-atomic routing games in which each link of the network is associated with a delay function, describing the dependence of the travel time on the traffic flow, and populations of agents independently select routes to minimize their individual travel times. The resulting equilibrium concept is the Wardrop equilibrium \cite{wardrop1952road}, defined as a flow where no traveler can reduce her travel time by unilaterally deviating to a different path. Under standard assumptions on the delay functions, the existence and uniqueness of Wardrop equilibria are well known \cite{beckmann1956studies}, and their efficiency has been extensively analyzed  \cite{roughgarden2002howbad, roughgarden2005selfishrouting}.

While classical models assume that all travelers behave independently and selfishly, modern transportation systems increasingly involve actors capable of coordinating the routing decisions of a subset of vehicles. In particular, delivery services and mobility providers operate fleets whose routes are centrally optimized according to fleet objectives, such as minimizing total travel time or improving service efficiency. In such settings, a coordinated fleet interacts with a large population of independent travelers who choose routes selfishly. This coexistence introduces a fundamental form of \emph{heterogeneity} in routing behavior, as different groups of users may pursue different objectives and respond differently to congestion~\cite{chen2017optimal,lazar2018efficient,toso2024coordinated,nilsson2018assignment,cianfanelli2022stability,cianfanelli2025stability}. 
At the same time, the growing diffusion of digital mobility platforms and navigation applications naturally leads to routing problems with {\em private information}, where the state of the world is uncertain, and agents make decisions based on probabilistic beliefs. Such problems can be modeled using \emph{Bayesian routing games}, where travel times depend on an uncertain state of the world and agents select routes according to their beliefs about that state \cite{harsanyi1967games,massicot2022competitive,ambrogio2025optimality}. The role of information in routing decisions has been widely studied recently, particularly in connection with the design of information provision mechanisms \cite{ambrogio2025optimality,das2017reducing,tavafoghi2017informational,cianfanelli2023information,zhu2022information}.

In this paper, we study a heterogeneous non-atomic routing game on a transportation network with two interacting populations: a coordinated fleet and a selfish population. The fleet coordinator selects the routes of its vehicles in order to optimize the total travel time of the fleet, while the selfish population selects routes according to Wardrop equilibrium behavior, minimizing individual travel costs, given their beliefs about the state of the world. 
The network is subject to an uncertain state that affects the delay functions, capturing stochastic events such as incidents, weather conditions, or roadworks. In general, we assume that neither the fleet nor the selfish users directly observe the state of the world: instead, they receive private signals that are correlated with it and among themselves. 
Assuming a commonly known prior distribution, we prove existence and uniqueness of Bayesian equilibria for these routing games. The importance of uniqueness in real-world applications can be appreciated, for instance, in the context of information design. Indeed, in a situation in which a designer wants to influence the users of a network through information, uniqueness enhances robustness and predictability, because the designer can anticipate a single outcome in response to a given policy or information structure.

The paper is organized as follows. In Section 2, we describe the transportation network model and the notion of Bayesian equilibria. In Section 3, we formulate the Bayesian routing game on the transportation network, we prove the existence of Bayesian equilibria, and provide sufficient conditions for their uniqueness. In Section 4, we summarize our findings.

\section{Model}

\subsection{Transportation network}
The \emph{transportation network} is modeled by a finite directed multigraph $\mathcal{G}=\left(\mathcal{V},\mathcal{E}, \sigma, \xi \right)$, where $\mathcal{V}$ is the set of nodes, $\mc E$ is the set of links, and $\sigma, \xi : \mathcal{E} \rightarrow \mathcal{V}$ are two functions that associate each link with its tail and head node, respectively. We assume that there are no self-loops, i.e., $\sigma(e)\ne\xi(e)$ for every $e$ in $\mc E$. 
 We assume that there is a unique origin and destination $o,d$ in $\mc V$, with $o \neq d$ and $d$ reachable from $o$, and let $\mc P$ denote the set of paths from $o$ to $d$. A path between two nodes $i$ and $j$ in $\mc V$ is a finite sequence of links $(e_1,...,e_n)$ such that $\sigma(e_1)=i$, $\sigma(e_{l+1})=\xi(e_{l})$ for every $l=1,\cdots,n-1$, and $\xi(e_n)=j$. The transportation network topology is then fully characterized by its node-link incidence matrix $B$ in $\R^{\mathcal V \times \mathcal E}$, with entries $B_{ie} = 1$ if $\sigma(e)=i$, $B_{ie} =-1$ if $\xi(e)=i$, or $B_{ie} =0$ otherwise. Moreover, we define the link-path incidence matrix $A$ in $\{0,1\}^{\mathcal E \times \mc P}$, with entries $A_{ep} = 1$ if link $e$ belongs to path $p$, or $A_{ep} = 0$ otherwise.


We consider two populations, called \emph{coordinated fleet} and \emph{selfish population}. For simplicity, we assume that both populations have unitary mass, and define the exogenous inflow $\nu = \delta^{(o)}-\delta^{(d)}$ in $\R^{\mc V}$, where  $\delta^{(i)}$ is the vector with $1$ in the $i$-th entry $\delta_i^{(i)}$ and $0$ in all other positions.
Hence, the set of network flows for both populations is
\begin{equation}
	\mc F =\{f \in \R_+^{\mc E}: Bf = \nu\}.
\end{equation}

The \emph{state of the world} is described by a random variable $W$ defined on a probability space $(\Omega, \mathcal{A}, \P)$ and taking values in a nonempty finite set $\mc W$. We shall assume without loss of generality that $p_w:=\P(W=w) > 0$ for every $w$ in $\mc W$.
Each link of the network is endowed with a delay function $\tau_e : \mc W \times \R_+ \to \R_+$ that depends on the state of the world as well as on the aggregate flow over the link. The delay functions $\tau_e(W,f_e)$ are assumed to be $\mc C^2$ and non-decreasing in the aggregate flow $f_e$.

\subsection{Information}
We assume that both the fleet coordinator and all selfish users know the prior distribution $(p_w)_{w\in\mc W}$ of the state of the world, but in general do not directly observe its realization $W$. Rather, they observe some private messages, to be denoted by $M$ for the fleet coordinator and $\bar M_s$ for the selfish agent $s$ in $\mc S:=[0,1)$, that carry some (partial) information on the state of the world $W$, and make their routing decisions based on the observed messages. 

Formally, let $\mc M=\{1,\cdots,|\mc M|\}$ and $\bar{\mc M}=\{1,\cdots,|\bar{\mc M}|\}$ be two non-empty finite sets where the private messages observed by the fleet coordinator and, respectively, the selfish users, take values, and consider two maps  
$\pi:\mc W\to\Delta_{\mc M}$ and $\bar\pi:\mc W\times\mc M\to\Delta_{\bar{\mc M}}$ taking values in the simplices of probability vectors over $\mc M$ and $\bar{\mc M}$, respectively, where $\Delta_{\mc K} =\{x \in \R_+^{\mc K}: \1'x = 1\}$. Let also $V$ and  $\bar V$ be two random variables defined over the probability space $(\Omega, \mathcal{A}, \P)$, both uniformly distributed over the interval $[0,1)$, mutually independent, and independent from $W$. Then, let 
\be\label{def:M}M=\min\Big\{m\in\mc M:\,V\le\sum_{1\le l\le m}\pi_l(W)\Big\}\,,\ee
be the message observed by the fleet coordinator, and let 
\be\label{def:barM}\!\!\!\bar M_s=\min\!\Big\{\bar m\in\bar{\mc M}:\!\,s+\bar V-\lfloor s+\ov V\rfloor\le\!\!\sum_{1\le l\le \bar m}\!\!\bar\pi_l(W,M)\Big\}\ee
be the message received by every selfish agent $s$ in $\mc S$. 

The following result states that selfish users $s$ in $\mc S$ have identical conditional distribution on their received messages $\bar M_s$ given the pair $(W,M)$, and that the (random) fraction of selfish users that receive a message $\bar m$ in $\bar{\mc M}$ is exactly $\bar\pi_{\bar m}(W,M)$.

\begin{lemma}\label{lemma:pi}
\begin{enumerate} 
\item[(i)] The triple $(W,M,\bar M_s)$ has the same joint probability distribution 
\be\hspace{-1.0cm}\label{joint-dist}\P(W=w,M=m,\bar M_s=\bar m)=p_w\pi_m(w)\bar\pi_{\bar m}(w,m)\,,\ee
for every selfish user $s$ in $\mc S$; 
\item[(ii)]
the fraction of selfish users receiving message $\bar m$  is
$$
\bar\pi_{\bar m}(W,M) = \P(\bar M_0 = \bar m| W,M)\,,
$$
for every $\bar m$ in $\bar{\mc M}$.
\end{enumerate}
\end{lemma}
\begin{proof}(i) On the one hand, \eqref{def:M} implies that the conditional distribution of the message $M$ given the state of the world $W$ is 
$\P(M=m|W=w)=\pi_m(w)$. On the other hand, \eqref{def:barM} and the fact that, for every $s$ in $\mc S$, the random variable $s+\bar V-\lfloor s+\bar V\rfloor$ is uniformly distributed on the interval $[0,1)$ and independent from $W$ and $M$, imply that the conditional probability distribution of $\bar M_s$ given $W$ and $M$ is  $\P(\bar M_s=\ov m|W=w,M=m)=\bar\pi_{\bar m}(w,m)$. This readily implies \eqref{joint-dist}.

(ii) Let $H_{\bar m}=\sum_{l=1}^{\bar m}\bar\pi_l(M,W)-\bar V$ for all $\bar m$ in $\bar{\mc M}$. 
Then, from \eqref{def:barM}, the set of selfish users that receive message $\bar m$ is
$$
\ba{rcl}\!\!\!\left\{s:\!\bar M_s\!=\!\bar m\right\}&\!\!\!\!\!=\!\!\!\!\!&
\left\{\!s:H_{\bar m-1}\!+\!\bar V\!\!<s\!+\!\bar V\!-\!\lfloor s\!+\!\ov V\rfloor\le\!\!H_{\bar m}\!+\!\bar V\right\}\\
&\!\!\!\!\!=\!\!\!\!\!&[0,1)\cap\big((H_{\bar m-1}, H_{\bar m}]\cup (H_{\bar m-1}\!+\!1, H_{\bar m}\!+\!1]\big),\ea
$$
whose Lebesgue measure is $H_{\bar m}-H_{\bar m-1}=\bar\pi_{\bar m}(W,M)$.
\end{proof}

Thanks to Lemma \ref{lemma:pi}, we can characterize the posterior beliefs on the state of the world after receiving the private messages for both the fleet coordinator and the selfish users. In particular, Lemma \ref{lemma:pi}(i) implies that, every two selfish users $s$ and $\tilde s$ in $\mc S$ that observe the same message $\bar M_s=\bar M_{\tilde s}$ have the same posterior distribution on the pair $(W,M)$ of the state of the world and the fleet coordinator's message. 


\subsection{Bayesian equilibria}
In this section we describe the user behavior and characterize the notion of equilibrium. Towards this goal, we define
$$\mc X=\{x:\mc M\to\mc F\}, \quad X_e=x_e(M) \quad \forall e \in \mc E\,. \\[-3pt]$$
An element $x$ in $\mc X$ is a map from the message that the fleet coordinator receives to the flow that he allocates on the network, while $X_e$ is the random variable describing such flow as a function of $M$. 
For the selfish population, we define
$$\mc Y=\{y:\bar{\mc M}\to\Delta_{\mc P}\},$$
$$
Z_e^y = \sum_{p\in \mc P}\! A_{ep} \!\!\sum_{\bar m\in\bar{\mc M}}\! y_p(\bar m)\bar\pi_{\bar m}(W,M) \\
 = \E[(Ay(\bar M_0))_e|W,M].
$$
An element $y\in\mc Y$ associates with each message $\bar m\in\bar{\mc M}$ a distribution over paths, where $y_p(\bar m)$ represents the fraction of selfish agents choosing path $p$ upon receiving message $\bar m$.
Given $(W,M)$, the quantity $Z_e^y$ denotes the induced flow of the selfish population on link $e$. Specifically, it is obtained by aggregating, over all messages and paths, the fraction of agents receiving message $\bar m$ multiplied by the fraction of those agents selecting path $p$, and then projecting the resulting path flow onto the link space through the incidence matrix $A$. The set of admissible $Z$ is
$$
\mc Z = \{Z: \mc W \times \mc M \to \R_+^{\mc E}: Z = Z^y \ \text{for $y \in \mc Y$}\}\,.
$$

The goal of the fleet coordinator is to minimize the expected total travel time spent on the network by the fleet, given the flow of the selfish population, defined by
\be\label{eq:cost_fleet}
\Phi(x,Z) = \E\big[\sum_{e } X_e
\tau_e(W,X_e + Z_e)\big]\,.\\[-4pt]
\ee
On the other hand, each user of the selfish population aims to travel along a path from $o$ to $d$ that minimizes the expected travel time conditioned on the message that he received. This yields the following notion of equilibrium. 
\begin{definition}\label{def:mixed_eq}
A pair $(x,Z)$ is said to be a \emph{Bayesian equilibrium} if: 1)
\begin{equation}\label{pb:prob_C}
x \in \argmin_{\tilde{x}\in \mc X} \Phi(\tilde x, Z)\,;
\end{equation}
2) there exists $y$ in $\mc Y$ such that $Z = Z^y$, and for every message $\bar m$ in $\bar{\mathcal M}$ and path $p$ in $\mc P$, $y_p(\bar m)>0$ implies
\be\label{eq:wardrop}
\E\Big[\sum_{e \in \mc E} (A_{ep}-A_{eq}) \tau_e\big(W, X_e+Z_e\big)\big|\bar{\mc M} = \bar m\Big] \le 0 \\[-4pt]
\ee
for every path $q$ in $\mc P$.
\end{definition} 

In plain words, a pair $(x,Z)$ is an equilibrium if $x$ minimizes the expected total travel time of the fleet given $Z$, and if every path $p$ that is used by some selfish agents that receive message $\bar m$ has minimal expected travel time according to the posterior belief after receiving message $\bar m$ and given $x$. 
In the next section we shall investigate existence and uniqueness of such Bayesian equilibria.

\begin{remark}
Bayesian equilibria can be viewed as Bayes correlated equilibria where the private messages serve as a correlation device \cite{bergemann2016bayes}. However, in this paper we restrict our analysis to a special class of Bayes correlated equilibria, since by construction the fraction of selfish agents that receive each message is a deterministic function of the state of the world. Another loss of generality comes from the fact that the set of actions of the fleet coordinator is continuous, while the set of messages $\mc M$ that he can receive is finite.
\end{remark}



\begin{ex}\label{ex}
	Consider a network with two nodes and two parallel links, with delay functions $\tau_1(W, f_1) = 4f_1 + W$, $\tau_2(W, f_2) = 8f_2$. 
	The state of the world $W$ takes value in $\mc W=\{0,1\}$, with $p_0 = 1/4$, $p_1 = 3/4$.
	Suppose that the coordinated fleet and the selfish agents can receive one and two messages respectively, i.e., $\mc M=\{m\}$, $\bar{ \mc M}=\{1, 2\}$, with
	\vspace{-0.15cm}
	\begin{equation*}
			\bar \pi_{1}(0) = \bar \pi_2(0)=1/2, \quad
			\bar \pi_{1}(1)= 2/3,\quad
			\bar \pi_{2}(1)=1/3. \\[-3pt]
	\end{equation*}
	Hence, by Bayes' theorem, the posterior distribution about the state of the world for the selfish population are 
	\begin{equation*}
		\begin{matrix}
			\P(W=0 | \bar{ M}=1)= 1/5, & \P(W=0 | \bar{M}=2) = 1/3\,, \\[4pt]
			\P(W=1 | \bar{ M}=1)= 4/5, & \P(W=1 | \bar{M}=2) = 2/3\,. \\[-3pt]
		\end{matrix}
	\end{equation*}
	
	We now compute the set of Bayesian equilibria $(x, Z)$. Since the network has only two parallel links, we have $x_2=1-x_1$ and $Z_2=1-Z_1$.
	Moreover, since the coordinated fleet has only one message, $x$ is a vector with entries $x_1=\alpha$, $x_2=1-\alpha$. Instead, for the selfish population we define $y_1(1)=\beta$, $y_1(2)=\gamma$. Therefore,
	\begin{align*}
		W = 0 \implies Z_1 & =y_1(1) \pi_1(0)+y_1(2) \pi_2(0)= \ds (\beta+\gamma)/2,\\[2pt]
		W = 1 \implies Z_1& =y_1(1) \pi_1(1)+y_1(2) \pi_2(1)= \ds (2\beta+\gamma)/3.	\\[-15pt]
	\end{align*}
	
	We now apply Definition \ref{def:mixed_eq} to find the Bayesian equilibria $(x, Z)$. For the coordinated fleet, $\Phi(x,Z)$ takes the form
	\begin{align*}
		& \frac{1}{4}[\alpha\tau_1(0,\alpha + \frac{\beta+\gamma}{2})+(1-\alpha)\tau_2(0,2-\alpha -\frac{\beta+\gamma}{2})] \ + \\[3pt]
		+ \ & \frac{3}{4}[\alpha\tau_1(1,\alpha + \frac{2\beta+\gamma}{3})+(1-\alpha)\tau_2(1,2-\alpha -\frac{2\beta+\gamma}{3})]\,,
	\end{align*}
	whose unique stationary point is achieved in $\alpha$ such that
	\begin{equation}\label{eq_coordinated}
		24\alpha + 15\beta/2 + 9\gamma/2 - 93/4=0\,. \\[-3pt]
	\end{equation}
We search for internal equilibria for the selfish population by assuming that \eqref{eq:wardrop} is satisfied as an equality. This yields
	\begin{equation*}
	\begin{aligned}\E[\tau_1\big(W, \alpha+Z_1\big)-\tau_2\big(W, 1-\alpha+1-Z_1\big)\big|\bar{\mc M} = 1] & = 0\,, \\
	\E[\tau_2\big(W,1-\alpha+1-Z_1\big)-\tau_1\big(W, \alpha+Z_1\big)\big|\bar{\mc M} = 2] & = 0\,.
	\end{aligned}
	\end{equation*}
Using the posterior distributions and the expression of the delay functions, this yields
$$
			12\alpha +\frac{38}{5}\beta+\frac{22}{5}\gamma-\frac{76}{5}=0\,, \hspace{0.3cm}
			-12\alpha -\frac{22}{3}\beta-\frac{14}{3}\gamma+\frac{46}{3}=0.\\[-2pt]
$$
	Solving this equation together with \eqref{eq_coordinated} we obtain 
	$\alpha=2/3, \beta=5/12, \gamma=11/12$,
	and hence $(x,y)$ and the corresponding Bayesian equilibrium $(x,Z^y)$. Note that this is the unique internal equilibrium, but other equilibria may exist. The presence of other equilibria will be ruled out by Theorem~\ref{gen_result}.
\end{ex}

\section{Main results}
In this section we establish conditions for the existence and uniqueness of Bayesian equilibria. Towards this goal, we first define an auxiliary two-player game, whose Nash equilibria are equivalent to the Bayesian equilibria.

\subsection{Auxiliary 2-player game}
We start with a technical result that is instrumental to show the equivalence between Bayesian equilibria and Nash equilibria of a 2-player game.

\begin{lemma}\label{lemma:equivalence_S}
	$Z$ satisfies condition 2) in Definition \ref{def:mixed_eq} if and only if it is solution of
	\begin{equation}\label{pb:prob_S}
		Z \in \argmin_{\tilde{Z} \in \mc Z} \bar\Phi(\tilde Z,x)\,,
	\end{equation}
	where $\bar \Phi: \mc Z \times \mc X \to \R_+$ is defined by
	$$
	\bar \Phi(Z,x) = \E\Big[\sum_{e \in \mc E} \int_{0}^{X_e + Z_e}\!\!\!\tau_e(W, r) dr\Big].
	$$
	Moreover, $\bar \Phi (Z,x)$ is convex in $Z$.
\end{lemma}

\begin{proof}
	$\bar \Phi(Z,x)$ is convex in $Z$ because the delay functions are non-decreasing.
	Since $\bar{\mc M}$ is finite, we can write $y$ in matrix form by identifying $y_{mp} \equiv y_p(m)$. With this notation, $y$ is a matrix in $\R_+^{\bar{\mc M} \times \mc P}$ subject to $y\1 = \1$. Now, since by construction every $Z$ in $\mc Z$ admits $y$ in $\mc Y$ such that $Z = Z^y$, and since $Z^y$ is a linear function of $y$, we can define $\Psi(y,x) = \bar \Phi(Z^y,x)$ and note that $\Psi(y,x)$ is convex in $y$, since it is the composition of a convex function with a linear function.
	Moreover,
	\be\label{eq:equivalence}
	Z^y \in \argmin_{\tilde{Z} \in \mc Z} \bar\Phi(\tilde Z,x) \iff y \in \argmin_{\tilde{y} \in \mc Y} \Psi(y,x)\,.
	\ee
	We now prove that $y$ solves the right-most optimization problem in \eqref{eq:equivalence} if and only if $Z^y$ satisfies condition 2) in Definition \ref{def:mixed_eq}, concluding the proof.
	Because of the convexity of $\Psi(y,x)$ in $y$, the KKT conditions of the right-most optimization problem in \eqref{eq:equivalence} are necessary and sufficient for optimality (see Section~3.1 of \cite{beckmann1956studies} for the derivation), that is, $y$ is solution of the right-most problem in \eqref{eq:equivalence} if and only if
	\be\label{eq:optPhi}
	\frac{\partial \Psi}{\partial y_{\bar m p}} \le \frac{\partial \Psi}{\partial y_{\bar m q}}\,, \quad \forall q \in \mc P\,,
	\ee
	for every $\bar m$ in $\bar{\mc M}$, $p$ in $\mc P$ such that $y_{\bar mp} > 0$. Explicit computation yields
	$$
	\frac{\partial \Psi}{\partial y_{\bar m p}} = \E\Big[\sum_{e \in \mc E} A_{ep} \bar\pi_{\bar m}(W,M) \tau_e(W,X_e+Z_e)\Big].
	$$
	By definition of $\bar\pi_{\bar m}(W,M)$, it follows that
	$\P(W = w, M = m, \bar M = \bar m) 
	= \P(W = w, M = m) \bar\pi_{\bar m}(W,M)$.
	Hence, dividing such expression by $\P (\bar M_0 = \bar m)$, we get
	$$
	\frac{\partial \Psi}{\partial y_{\bar m p}} = \E\Big[\sum_{e \in \mc E} A_{ep} \tau_e(W,X_e+Z_e)\Big| \bar M = \bar m\Big]\cdot \P(\bar M_0 = \bar m).
	$$
	The proof is concluded by plugging this expression into \eqref{eq:optPhi} and noticing that this is equivalent to \eqref{eq:wardrop}.
\end{proof}

We now introduce an auxiliary two-players game.

\begin{definition}\label{def:game}
	We consider the game $\Gamma$ with player set $\{C,S\}$, action sets $\{\mc X,\mc Z\}$ and utility functions $\{u(x,Z),\bar u(Z,x)\}$, with
	$$
	u(x,Z) = -\Phi(x,Z), \quad \bar u(Z,x) = - \bar \Phi(Z,x).\\[3pt]
	$$ 
\end{definition}

Here, player $C$ represents the coordinating decision maker choosing the action $x \in \mc X$, while player $S$ is a nominal player representing the collection of all selfish agents in the original game. The artificial aggregate action $Z \in \mc Z$ is defined for technical reasons and it summarizes the collective response of the selfish agents. Hence, the two-player game captures the interaction between the coordinator and the overall behavior of the selfish population.

\begin{definition}\label{def:nash_eq}
	A \emph{pure-strategy Nash equilibrium} of game $\Gamma$ is an action distribution $(x,Z)$ such that
	\be\label{eq:best_response}
	x \in \argmax_{\tilde x \in \mc X} u(\tilde x, Z)\,,\quad
	Z \in \argmax_{\tilde Z \in \mc Z} \bar u(\tilde Z, x)\,.
	\ee
\end{definition}
\vspace{0.1cm}
\begin{prp}\label{prp:equivalence}
	$(x,Z)$ is a pure-strategy Nash equilibrium of $\Gamma$ if and only if it is a Bayesian equilibrium per Definition~\ref{def:mixed_eq}.
\end{prp}
\begin{proof}
	By definition of $U$, $x$ satisfies the left-most relation of \eqref{eq:best_response} if and only if it satisfies \eqref{pb:prob_C}. By Lemma \ref{lemma:equivalence_S}, $Z$  satisfies the right-most relation of \eqref{eq:best_response} if and only if it satisfies 2) in Definition \ref{def:mixed_eq}. This concludes the proof.
\end{proof}

\subsection{Existence of Bayesian equilibria}
The next result provides sufficient conditions on the delay functions for the existence of Bayesian equilibria.
Note that our contribution concerns a rich setting combining heterogeneous populations with different objectives, private information and Bayesian routing strategies conditioned on private messages. In this setting, the equilibrium is not a standard Wardrop equilibrium, but rather a coupled Bayesian equilibrium involving both an optimization problem for the coordinator and Wardrop-type conditions for the selfish population under uncertainty. The existence result is therefore not an immediate consequence of classical Wardrop theory.

\begin{prp}\label{exists_symm_info_finite}
	Assume that the delay functions $\tau_e(W,f_e)$ are convex in $f_e$ for every $W$ in $\mc W$. Then, there always exists a Bayesian equilibrium.
\end{prp}

\begin{proof}
	We first prove the existence of Nash equilibria of $\Gamma$.
	To do so, we rely on \cite{debreu1952social}, which ensures that a pure strategy Nash equilibrium of the game $\Gamma$ exists if: (i) the action sets are finite-dimensional, compact, and convex; (ii) the utility function of each player is continuous in both actions and concave in the action of the player itself.
	
	The action sets $\mc X,\mc Z$ are finite-dimensional (since $\mc W$ contains a finite number of elements), compact and convex. Moreover, the utility functions $u(x,Z)$ and $\bar u(Z,x)$ are continuous in both actions, since the delay functions are continuous. We now aim to prove that $u(x,Z)$ is concave in $x$ and that $\bar u(Z,x)$ is concave in $Z$. To prove the concavity of $u(x,Z)$, note that
	$$
	\begin{aligned}
		u(x,Z) &= - \sum_{w \in \mc W}\sum_{m \in \mc M} \P(W = w, M=m) \ \cdot \\
		& \cdot \sum_{e \in \mc E} x_e(m) \tau_e(w,x_e(m) + Z_e)\,,\\[-7pt]
	\end{aligned}  
	$$
	hence
	$$
	\frac{\partial u(x,Z)}{\partial x_{e}(m) \partial x_{\tilde e}(\tilde m)} = 0, \quad \text{if} \ e \neq \tilde e \ \text{or} \ m \neq \tilde m.\\[-2pt]
	$$
	Therefore, the Hessian of $U$ is diagonal. Direct computation yields that the diagonal elements are non-positive if the delay functions $\tau_e(W,f_w)$ are non-decreasing and convex in $f_e$. This proves the concavity of $u(x,Z)$ in $x$. The concavity of $\bar u(Z,x)$ in $Z$ follows from $\bar u(Z,x) = - \bar \Phi(Z,x)$ and from convexity of $\bar \Phi(Z,x)$ in $Z$, established by Lemma \ref{lemma:equivalence_S}.
	
	This proves the existence of a Nash equilibrium of the two-players game $\Gamma$. The existence of a Bayesian equilibrium then follows from Proposition \ref{prp:equivalence}.
\end{proof}

\subsection{Uniqueness}
This section provides sufficient conditions for the uniqueness of Bayesian equilibria.

\begin{theorem}\label{gen_result}
	Assume that the delay functions are strictly increasing and convex, and
	\begin{equation}\label{eq:uniqueness}
		2\tau'_e(W,f_e + \bar f_e) > f_e \tau_e''(W,f_e + \bar f_e) 
	\end{equation}
	for every $e$ in $\mc E$, $W$ in $\mc W$, and $f,\bar f$ in $\mc F$. Then, there exists a unique Bayesian equilibrium.
\end{theorem}

\begin{proof}
	We prove that the auxiliary game $\Gamma$ admits a unique Nash equilibrium. Then, the uniqueness of the Bayesian equilibrium follows from Proposition \ref{prp:equivalence}.
	
	Since $Z$ depends on $(W,M)$, and since $W,M$ can take a finite number of values, we can express the elements $Z$ in $\mc Z$ in matrix form. Let $Q=(W,M)$ take values in $\mc Q = \mc W\times\mc M$. Since the set of actions $\mc Z$ of the selfish population is the image of a linear function $Z^y$, for every $y$ in $\mc Y$, standard arguments of linear algebra imply the existence of two matrices $H,K$ such that
	$$
	\mc Z = \{Z \in \R_+^{\mc Q \times E}: HZ = K\}\,.
	$$
	For convenience, we also represent $\mc X$ in matrix form by
	$$
	\mc X = \{x \in \R_+^{\mc M \times \mc E}: \sum_{e} B_{ie} x_{me} = \nu_i, \ \forall i \in \mc V, \ \forall m \in \mc M\}\,.
	$$
	For $q = (w,m)$ in $\mc Q$, let $q_1 = w$, $q_2 = m$.
	It will also prove useful to embed $\mc X$ into the augmented set
	\begin{align*}
		\hat{\mc X} = \{& \hat x \in \R_+^{\mc Q \times \mc E}: \sum_{e}B_{ie} \hat x_{qe} = \nu_i, \ \forall i \in \mc V, \ \forall q \in \mc Q\,, \\
		& \hat x_{q e} = \hat x_{q'e}\,, \forall e \in \mc E, \forall q, q' \in \mc Q \text{ s.t. } q_2 = q'_2\}\,.
	\end{align*}
	In plain words, $\hat{\mc X}$ is a new artificial set of actions for the fleet coordinator, where the action is in principle allowed to depend explicitly on the state of the world, under the constraint that this dependence is fictitious, as the actions are constrained to depend only on the message. Notice that $\mc X$ can be embedded into $\hat{\mc X}$, and as a consequence, there exists a bijection between $\mc X$ and $\hat{\mc X}$. We also parametrize the utility functions of the two players in order to let them depend on $x$ in $\hat{\mc X}$, $Z$ in $\mc Z$. Let $v, \bar v: \hat{\mc X} \times \mc Z \to \R$ be defined by
	$$
	v(x,Z) = - \sum_{q \in \mc Q} \P(Q = q) \sum_{e \in \mc E} x_{qe}
	\tau_e(q_1, x_{qe} + Z_{qe})\,,
	$$
	$$
	\bar v(x,Z) = - \sum_{q \in \mc Q} \P(Q = q) \sum_{e \in \mc E} \int_0^{x_{qe} + Z_{qe}}
	\tau_e(q_1, r) \mathrm{d} r\,.
	$$
	We have now defined a new two-players game $\hat \Gamma$ with action sets $\hat{\mc X},\mc Z$ and utility functions $v,\bar v$, whose Nash equilibria $(x,Z)$
	satisfy
	\be\label{eq:nash2}
	x \in \argmax_{\tilde x \in \hat{\mc X}} v(\tilde x, Z)\,,\quad
	Z \in \argmax_{\tilde Z \in \mc Z} \bar v(\tilde Z, x)\,.
	\ee
	It is fundamental to observe that, by construction, there is a bijection between Nash equilibria of $\hat \Gamma$ and those of $\Gamma$, so that $\Gamma$ admits a unique Nash equilibrium if and only $\hat \Gamma$ admits a unique Nash equilibrium.
	
	We now aim to prove the uniqueness of the Nash equilibrium of $\hat\Gamma$. The Lagrangian function of the left-most optimization problem in \eqref{eq:nash2} is
	$$
	\begin{aligned}
		L(x,Z,\gamma,\lambda) & = v(x,Z) + \sum_{e,q} \gamma_{qe} x_{qe} + \\[-3pt]
		& + \sum_{i,q} \lambda_{iq} \big(\sum_{e} B_{ie} x_{qe} - \nu_i\big),\\[-3pt]
	\end{aligned}
	$$
	where $\gamma_{eq} \ge 0$ and $\lambda_{iq}$ are Lagrangian multipliers associated to inequality constraint $x_{qe}\geq 0$ and equality constraints $\sum_{e}B_{ie} x_{qe} = \nu_i$ that an element $x$ in $\hat{\mc X}$ must satisfy.
	The Lagrangian function of the right-most optimization in \eqref{eq:nash2} is
	$$
	\begin{aligned}
		\bar L(Z,x,\omega,\delta) & = \bar v(Z,x) + \sum_{q,e} \omega_{q e} Z_{q e} + \\[-3pt]
		& + \sum_{q',e} \delta_{q'e} \big(\sum_{q} H_{q'q} Z_{q e} - K_{q'e}\big),\\[-3pt]
	\end{aligned}
	$$
	where $\omega_{qe} \ge 0$ and $\delta_{q' e}$ are Lagrangian multipliers associated to inequality constraint $Z_{qe} \geq 0$ and equality constraints $\sum_{q} H_{q'q} Z_{q e} = K_{q'e}$ of $\mc Z$, respectively.
	
	Suppose now that there are two different Nash equilibria $(x,Z) \neq (\tilde x, \tilde Z)$. For $(x,Z)$ to be a Nash equilibrium, there must exist a set of Lagrangian multipliers $(\gamma,\lambda,\omega,\delta)$ such that
	$$
	\frac{\partial L(x,Z,\gamma,\lambda)}{\partial x_{qe}} = 0\,, \quad \frac{\partial \bar L(Z,x,\omega,\delta)}{\partial z_{qe}} = 0\,,\\[-3pt]
	$$
	for every $e$ in $\mc E$, $q$ in $\mc Q$, which readily implies
	\begin{equation}\label{eq:1}
		\begin{split}
			&\frac{\partial v(x,Z)}{\partial x_{qe}} + \gamma_{qe} +\sum_{i} \lambda_{iq} B_{ie} = 0\,, \\[-3pt]
			&\frac{\partial \bar v(Z,x)}{\partial Z_{qe}} + \omega_{qe} +\sum_{q'} \delta_{q'e} H_{q'q} = 0\,.\\[-8pt]
		\end{split}
	\end{equation}
	Equivalently, for $(\tilde x,\tilde Z)$ to be a Nash equilibrium, there must exist a set of Lagrangian multipliers $(\tilde \gamma,\tilde \lambda,\tilde \omega,\tilde \delta)$ such that
	\begin{equation}\label{eq:2}
		\begin{split}
			&\frac{\partial v(\tilde x,\tilde Z)}{\partial \tilde x_{qe}} + \tilde \gamma_{qe} +\sum_{i} \tilde \lambda_{iq} B_{ie} = 0\,, \\
			&\frac{\partial \bar v(\tilde Z,\tilde x)}{\partial \tilde Z_{qe}} + \tilde \omega_{qe} +\sum_{q'} \tilde \delta_{q'e} H_{q'q} = 0\,.\\[-8pt]
		\end{split}
	\end{equation}
	Moreover, the complementary slackness constraints, i.e.,
	\be\label{eq:slackness}
	\!\omega_{qe} Z_{qe} = 0\,, \  \gamma_{qe} x_{qe} = 0\,, \ \tilde\omega_{qe} \tilde Z_{qe} = 0\,, \  \tilde \gamma_{qe} \tilde x_{qe} = 0\,,
	\ee
	must hold true for every $e$ in $\mc E$, $q$ in $\mc Q$.
	We now multiply the first equation of \eqref{eq:1} by $(\tilde x_{qe}-x_{qe})$, the second equation of \eqref{eq:1} by $(\tilde Z_{qe}- Z_{qe})$, the first equation of \eqref{eq:2} by $(x_{qe}- \tilde x_{qe})$ and the second equation of \eqref{eq:2} by $(Z_{qe}- \tilde Z_{qe})$.
	Then, summing all the resulting terms for every $e,q$ and rearranging them we get that 
$\psi+\chi=0,$
	where
	\begin{equation*}
		\begin{aligned}
			\psi & = \sum_{q, e} (\tilde x_{qe}- x_{qe})\Big(\frac{\partial v(x,Z)}{\partial x_{qe}} - \frac{\partial v(\tilde x,\tilde Z)}{\partial \tilde x_{qe}}\Big) \ + \\[-3pt]
			& + \sum_{q,e} (\tilde Z_{qe}- Z_{qe})\Big(\frac{\partial \bar v(Z,x)}{\partial Z_{qe}} - \frac{\partial \bar v(\tilde Z,\tilde x)}{\partial \tilde Z_{qe}}\Big)\,,
		\end{aligned}
	\end{equation*}
	\begin{equation*}
		\begin{aligned}
			\chi & = \!\sum_{q, e} (\tilde x_{qe} \!-\! x_{qe})\big(\gamma_{qe}\!-\!\tilde \gamma_{qe} \! + \! \sum_{i} \!\! \big(\lambda_{iq} B_{ie} \!-\! \tilde \lambda_{iq} B_{ie}\big)\big) \ + \\
			& + \!\sum_{q,e} (\tilde Z_{qe}\!-\! Z_{qe}) \big(\omega_{qe} \!-\! \tilde \omega_{qe} \! +\! \sum_{q'}\!\! \big(\delta_{q'e} H_{q'q} \!-\! \tilde \delta_{q'e} H_{q'q}\big)\big)\,.
		\end{aligned}
	\end{equation*}
	We now analyze $\chi$ and $\psi$ separately, starting from $\chi$.
%
	First, note that, for $x$ and $\tilde x$ to belong to $\hat{\mc X}$, and for $Z$ and $\tilde Z$ to belong to $\mc Z$,
	$$
	\begin{aligned}
	\sum_e B_{ie} x_{qe} = \sum_e B_{ie} \tilde x_{qe} = \nu_i, \quad \forall i,q\,.\\
	\sum_{q} H_{q'q} Z_{qe} = \sum_{q} H_{q'q} \tilde Z_{qe} = K_{q'e}\,, \quad \forall q',e\,.
	\end{aligned}
	$$
	Using these two equations and \eqref{eq:slackness}, we get
	$$
	\chi = \sum_{q,e} (\tilde{x}_{qe} \gamma_{qe} + x_{qe} \tilde{\gamma}_{qe}) + \sum_{q,e} (\tilde{Z}_{qe} \omega_{qe} + Z_{qe} \tilde \omega_{qe}) \ge 0,
	$$
	where the last equivalence follows since both the Lagrangian multipliers $\gamma,\omega$ and the actions $x,Z$ are non-negative.
	
	We now prove that $\psi \ge 0$ and $\psi = 0$ if and only if the two equilibria coincide. Note that $(x,Z) = (\tilde x,\tilde Z)$ implies $\psi = 0$. We are then left to prove that $(x,Z) \neq (\tilde x,\tilde Z)$ implies $\psi>0$.
	To this end, we define $\sigma(x,Z): = v(x,Z)+\bar v(Z,x)$ and its pseudogradient 
	\begin{equation*}
		g(x,Z) :=
		\bigg[ \begin{matrix}
			\nabla_x v(x,Z) \\
			\nabla_Z \bar v(Z,x)
		\end{matrix} \bigg]\,,
	\end{equation*}
	obtained by stacking the gradient of the utility function of both players with respect to their actions. As shown in \cite[Section 3]{rosen1965existence}, if $(x,Z) \neq (\tilde x,\tilde Z)$ and $\sigma(x,Z)$ is diagonally strictly concave for every $(x,Z)$, then $\psi > 0$. The same reference shows that a sufficient condition for $\sigma(x,Z)$ to be diagonally strictly concave is that the Jacobian $J(x,Z)$ of the pseudogradient $g(x,Z)$ is negative definite for every $x,Z$. Direct computation yields that $J(x,Z)$ has a block diagonal structure, since $\partial^2 V / (\partial x_{qe}\partial x_{q'e'}) = \partial^2 V / (\partial x_{qe}\partial Z_{q'e'}) = \partial^2 \bar V / (\partial Z_{qe}\partial Z_{q'e'}) = \partial^2 \bar V / (\partial Z_{qe}\partial Z_{q'e'}) = 0$ if $(q,e) \neq (q',e')$. Moreover,
	\begin{align*}
		&\frac{\partial^2 v(x,Z)}{\partial x_{qe}^2} = - \P(Q = q) (2\tau_e'+x_{qe}\tau_e'')\\
		&\frac{\partial^2 v(x,Z)}{\partial x_{qe} \partial Z_{qe}}= -\P(Q = q)(\tau_e'+x_{qe} \tau_e'')  \\
		&\frac{\partial^2 \bar v(Z,x)}{\partial Z_{qe}\partial x_{qe}} = \frac{\partial^2 \bar v(Z,x)}{\partial Z_{qe}^2} = -\P(Q = q)\tau_e'\,.
	\end{align*}
	We are now left to prove that, for every $e$ in $\mc E$ and $q$ in $Q$, the $(e,q)$-th block of $J$ is negative definite, so that the entire Jacobian is negative definite. To do so, we sum the block with its transpose, obtaining
	$$
	-\P(Q = q) \left(\ba{cc} 2(2\tau_e' + x_{qe} \tau_e'') & 2\tau_e' + x_{qe} \tau_e'' \\
	2 \tau_e' + x_{qe} \tau_e'' & 2\tau_e' \ea\right)\,.
	$$
	Since the delay functions are strictly increasing and convex, the trace of the matrix is negative, while the determinant is positive if and only if
	\begin{equation*}
		2\tau'_e(W,f_e + \bar f_e) > f_e \tau_e''(W,f_e + \bar f_e) \,, \quad \forall q \in \mc Q, \ \forall e \in \mc E\,.
	\end{equation*}
	This proves the uniqueness of the Nash equilibrium of $\hat \Gamma$, and therefore of $\Gamma$. Proposition \ref{prp:equivalence} then implies the uniqueness of the Bayesian equilibrium.
\end{proof}

\begin{remark}
A sufficient condition similar to \eqref{eq:uniqueness} was established by \cite{nilsson2018assignment} for uniqueness of the equilibria in presence of two populations in a deterministic setting. It is also worth noting that the uniqueness of the Nash equilibrium of $\Gamma$, which is an intermediate step to proving Theorem \ref{gen_result}, is established using an argument similar to that in \cite[Theorem 2]{rosen1965existence}. In that work, sufficient conditions for uniqueness are derived for the case in which the action sets are subject to inequality constraints, but without any equality constraints.
\end{remark}

\section{Conclusion}
We have established existence and uniqueness of Bayesian equilibria for non-atomic Bayesian routing games, whereby a coordinated fleet and a population of selfish users share a transportation network. We have assumed that the travel times on the links depend on both the aggregate flow on them and a random state of the world that is generally not directly observable by neither the fleet coordinator nor the selfish users, who make routing decisions based on private partial information obtained from messages that are correlated with the state of the world and among themselves. Future directions include extending the analysis to continuous sets of states of the world $\mc W$ and sets of messages $\mc M$, as well as studying information design problems, where a system planner designs the private messaging policy aiming at optimizing the social welfare.


\bibliography{biblio}

@inproceedings{nilsson2018assignment,
	author = {Nilsson, G. and Grover, P. and Kalabi{\'c}, U.},
	booktitle = {57th IEEE CDC},
	pages = {1023--1028},
	title = {Assignment and control of two-tiered vehicle traffic},
	year = {2018}}

@article{rosen1965existence,
	author = {Rosen, J.B.},
	journal = {Econometrica},
	pages = {520--534},
	title = {Existence and uniqueness of equilibrium points for concave n-person games},
	year = {1965}}

@article{debreu1952social,
	author = {Debreu, G.},
	journal = {Proceedings of the national academy of sciences},
	number = {10},
	pages = {886--893},
	publisher = {National Academy of Sciences},
	title = {A social equilibrium existence theorem},
	volume = {38},
	year = {1952}}

@article{ambrogio2025optimality,
  title={On Optimality of Private Information in {Bayesian} Routing Games},
  author={Ambrogio, A. and Cianfanelli, L. and Como, G.},
  journal={arXiv preprint arXiv:2509.03357},
  year={2025}
}

@article{roughgarden2002howbad,
	author = {Roughgarden, T. and Tardos, {\'E}.},
	journal = {Journal of the ACM},
	number = {2},
	pages = {236--259},
	title = {How Bad is Selfish Routing?},
	volume = {49},
	year = {2002}}

@book{roughgarden2005selfishrouting,
	author = {Roughgarden, T.},
	publisher = {MIT},
	title = {Selfish Routing and the Price of Anarchy},
	year = {2005}}

@article{chen2017optimal,
	author = {Chen, X. and He, X. and Yin, Y.},
	journal = {Transportation Research Part B: Methodological},
	pages = {620--634},
	title = {Optimal Deployment of Autonomous Vehicles in Traffic Networks},
	volume = {105},
	year = {2017}}

@inproceedings{lazar2018efficient,
	author = {D.A. Lazar and A. Khani and G. Zussman},
	booktitle = {17th AAMAS},
	pages = {2191--2193},
	title = {Efficient Traffic Routing in the Presence of Mixed Autonomy},
	year = {2018}}

@article{harsanyi1967games,
	author = {Harsanyi, J.C.},
	journal = {Management Science},
	number = {3,5,7},
	pages = {159--182, 320--334, 486--502},
	title = {Games with Incomplete Information Played by ``{B}ayesian'' Players, Part I--III},
	volume = {14},
	year = {1967}}

@inproceedings{cianfanelli2023information,
	author = {Cianfanelli, L. and Ambrogio, A. and Como, G.},
	booktitle = {2023 62nd IEEE Conference on Decision and Control (CDC)},
	organization = {IEEE},
	pages = {3945--3949},
	title = {Information design in {B}ayesian routing games},
	year = {2023}}

@article{wardrop1952road,
	author = {Wardrop, J.G.},
	journal = {Proceedings of the Institution of Civil Engineers},
	number = {3},
	pages = {325--362},
	publisher = {Thomas Telford-ICE Virtual Library},
	title = {Some theoretical aspects of road traffic research},
	volume = {1},
	year = {1952}}

@inproceedings{das2017reducing,
	author = {Das, S. and Kamenica, E. and Mirka, R.},
	booktitle = {55th Allerton Conf.},
	pages = {1279--1284},
	title = {Reducing congestion through information design},
	year = {2017}}

@article{massicot2022competitive,
	author = {Massicot, O. and Langbort, C.},
	journal = {IEEE Transactions on Control of Network Systems},
	number = {4},
	pages = {1589-1599},
	title = {Competitive Comparisons of Strategic Information Provision Policies in Network Routing Games},
	volume = {9},
	year = {2022}}

@inproceedings{tavafoghi2017informational,
	author = {H. Tavafoghi and D. Teneketzis},
	booktitle = {55th Allerton Conf.},
	pages = {1285--1292},
	title = {Informational incentives for congestion games},
	year = {2017}}

@inproceedings{toso2024coordinated,
	author = {T. Toso and F. Parise and P. Frasca and A.Y. Kibangou},
	booktitle = {63rd IEEE CDC},
	doi = {10.1109/CDC56724.2024.10886835},
	pages = {4173-4178},
	title = {On the impact of coordinated fleets size on traffic efficiency},
	year = {2024}}

@techreport{beckmann1956studies,
  title={Studies in the Economics of Transportation},
  author={Beckmann, M. and McGuire, C.B. and Winsten, C.B.},
  year={1956}
}

@article{zhu2022information,
	author = {Zhu, Y. and Savla, K.},
	journal = {IEEE Transactions on Control of Network Systems},
	number = {2},
	pages = {613--624},
	publisher = {IEEE},
	title = {Information design in nonatomic routing games with partial participation: Computation and properties},
	volume = {9},
	year = {2022}}

@article{bergemann2016bayes,
  title={Bayes correlated equilibrium and the comparison of information structures in games},
  author={Bergemann, D. and Morris, S.},
  journal={Theoretical Economics},
  volume={11},
  number={2},
  pages={487--522},
  year={2016},
  publisher={Wiley Online Library}
}

@inproceedings{cianfanelli2022stability,
  title={Stability and bifurcations in transportation networks with heterogeneous users},
  author={Cianfanelli, L. and Como, G. and Toso, T.},
  booktitle={2022 IEEE 61st Conference on Decision and Control (CDC)},
  pages={6371--6376},
  year={2022},
  organization={IEEE}
}

@article{cianfanelli2025stability,
  title={On the stability of the logit dynamics in population games},
  author={Cianfanelli, L. and Como, G.},
  journal={IEEE Transactions on Automatic Control},
  volume={70},
  number={9},
  pages={5910--5925},
  year={2025},
  publisher={IEEE}
}
\bibliographystyle{plain}

\end{document}